\documentclass[a4paper, english, 11pt]{article}
\pdfoutput=1

\usepackage[utf8]{inputenc}
\usepackage[T1]{fontenc}
\usepackage[english]{babel}
\usepackage{csquotes}
\usepackage{textcomp}
\usepackage{varioref}
\usepackage{amsmath}
\usepackage{amssymb}
\usepackage{amsthm}
\usepackage{graphicx}
\usepackage[section]{placeins}
\usepackage{listings}
\usepackage{color}
\usepackage{titlesec}
\usepackage{tikz}
\usetikzlibrary{patterns,decorations}
\usepackage{algpseudocode}
\usepackage{algorithmicx}
\usepackage{algorithm}
\usepackage{hyperref}
\usepackage{enumitem}
\usepackage{subfig}
\usepackage{longtable}
\usepackage{caption}
\usepackage{bookmark}
\usepackage{multirow}
\usepackage{booktabs}

\definecolor{dkgreen}{rgb}{0,0.6,0}
\definecolor{gray}{rgb}{0.5,0.5,0.5}
\definecolor{mauve}{rgb}{0.58,0,0.82}

\makeatletter
\newtheorem*{rep@theorem}{\rep@title}
\newcommand{\newreptheorem}[2]{%
\newenvironment{rep#1}[1]{%
 \def\rep@title{#2 \ref{##1}}%
 \begin{rep@theorem}}%
 {\end{rep@theorem}}}
\makeatother

\theoremstyle{plain}
\newtheorem{mytheorem}{Theorem}
\newtheorem{myprop}[mytheorem]{Proposition}
\newreptheorem{myprop}{Proposition}
\newtheorem{mylemma}[mytheorem]{Lemma}

\theoremstyle{definition}
\newtheorem{mydef}{Definition}

\newtheorem*{myremarks}{Remarks}

\newcommand{\Unif}{\text{\normalfont{Unif}}}

\newcommand{\bv}{\mathbf{v}}

\newcommand{\bx}{\mathbf{x}}

\newcommand{\by}{\mathbf{y}}

\newcommand{\bmu}{\boldsymbol{\mu}}

\newcommand{\bSigma}{\boldsymbol{\Sigma}}

\makeatletter
\newcommand*{\tp}{%
{\mathpalette\@tp{}}%
}
\newcommand*{\@tp}[2]{%
\raisebox{\depth}{$\m@th#1\intercal$}%
}
\makeatother

\algnewcommand\algorithmicinput{\textbf{Input:}}
\algnewcommand\INPUT{\item[\algorithmicinput]}

\algnewcommand\algorithmicmyreturn{\textbf{Return:}}
\algnewcommand\myRETURN{\item[\algorithmicmyreturn]}

\usepackage[backend    = biber, 
            style      = authoryear, 
            giveninits = true,
            isbn       = false,
            url        = false,
            doi        = false,
            eprint     = false,
            sorting    = nyt, 
            natbib     = true]{biblatex}
\renewbibmacro{in:}{%
  \ifentrytype{article}{}{\printtext{\bibstring{in}\intitlepunct}}}
\makeatletter
    \newrobustcmd*{\nobibliography}{%
      \@ifnextchar[
        {\blx@nobibliography}
        {\blx@nobibliography[]}}
    \def\blx@nobibliography[#1]{}
    \appto{\skip@preamble}{\let\printbibliography\nobibliography}
\makeatother

\usepackage[margin=1in]{geometry}

\newcommand{\blind}{0}

\def\spacingset#1{\renewcommand{\baselinestretch}%
{#1}\small\normalsize} \spacingset{1}
\spacingset{1} 

\begin{document}
\if1\blind
{
  \bigskip
  \bigskip
  \bigskip
  \begin{center}
    {\LARGE\bf Which principal components are most sensitive to distributional changes?} 
  \end{center}
  \medskip
} \fi

\if0\blind
{
  \title{\bf Which principal components are most sensitive to distributional changes?} 
  \author{Martin Tveten,
  Department of Mathematics, University of Oslo}
  \maketitle
} \fi

\begin{abstract}
PCA is often used in anomaly detection and statistical process control tasks.
For bivariate data, we prove that the minor projection (the least varying projection) of the PCA-rotated data is the most sensitive to distributional changes,
where sensitivity is defined by the Hellinger distance between distributions before and after a change.
In particular, this is almost always the case if only one parameter of the bivariate normal distribution changes, i.e., the change is sparse.
Simulations indicate that the minor projections are the most sensitive for a large range of changes and pre-change settings in higher dimensions as well.
This motivates using the minor projections for detecting sparse distributional changes in high-dimensional data.
\end{abstract}

\newpage

\section{Introduction}
It is popular to use principal component analysis (PCA) for anomaly detection and stochastic process control (SPC).
Using PCA in SPC goes back to the work of \citet{jackson_application_1957} and \citet{jackson_control_1979},
and its various extensions (see \citet{ketelaere_overview_2015} and \citet{rato_systematic_2016} for an overview) have been succesfully applied to many real data situations.
Within the machine learning litterature on anomaly detection, \citet{mishin_real_2014} use PCA for temperature monnitoring at Johns Hopkins, 
\citet{harrou_improved_2015} apply PCA-based anomaly detection to find segments with abnormal rates of patient arrivals at an emergency department, 
and \citet{camacho_pca-based_2016} relate PCA-based monitoring in SPC with modern anomaly detection in statistical networks.
\citet{pimentel_review_2014} provide an extensive review of novelty detection techniques and applications, and it is pointed to PCA being very useful for detecting outliers in this setting, 
for a large range of real world examples, covering industrial monitoring, video surveilance, text mining, sensor networks and IT security.
In this review, as well as in \citet{lakhina_diagnosing_2004} and \citet{huang_-network_2007}, it is acknowledged that it is most often the residual subspace of PCA that is most useful for outlier detection.
\citet{kuncheva_pca_2014} study PCA in the related field of change detection, but with a fascinating twist.

In general, most PCA-based methods utilize PCA in the intended way of creating a model based on retaining a small number of the most varying projections onto eigenvectors of the correlation matrix.
The data is thus split into a model subspace that explains most of the variance in the data and a residual subspace.
Rather than using PCA in this regular way for change detection as well, 
\citet{kuncheva_pca_2014} pose the question of which projections/principal components are the \textit{most sensitive to distributional changes} in the data. 
Sensitivity is measured by a statistical distance between the marginal distributions of projections before and after a change.
They give a brief two-dimensional theoretical example that motivates monitoring the minor projections (the least varying projections) to detect anomalies.
An important feature of such an approach is that it can be used to choose a subspace based on criteria linked to change detection, rather than on retaining information/variance in the data,
potentially yielding better change and anomaly detection methods.
The goal of this article is to give a more complete treatment of the bivariate problem of \citet{kuncheva_pca_2014}, to understand the projection's sensitivity under a simple setup,
and then study how these results carry over to higher dimensions by simulations.

There are three main differences between our approach and the approach of \citet{kuncheva_pca_2014}.
First of all, \citet{kuncheva_pca_2014} look at changes in the projections directly, without connecting them to which changes in the original data that cause them.
Since it is changes in the distribution of the original data that are of interest,
the propositions one actually wants to obtain concern changes in the original data, not the resulting ones in the projections.
We trace the changes in the original data through the projections, to see how the distributions of the projections change as a consequence.
For instance, we give an answer to questions like "if the variances of the original components double, how does that affect the projections, and which of the projections will change the most?",
as opposed to "if the variances of the projections both increase by adding 1, which of the projections change that most?" of \citet{kuncheva_pca_2014}.
In the latter case, clearly the projection with the lowest variance will change the most relative to what it was,
and we have gained no knowledge about how the original data must change to achieve this change in the projections.

Secondly, we study a much larger space of possible changes, including changes in only one parameter at a time.
Such change scenarios where only a few of the dimensions change we call \textit{sparse changes}.
This turns out to be the main gist of this work, since it is especially in the sparse change scenarios that the minor projection is the most sensitive.
Hence, we obtain clearer hints of which change scenarios the minor components might be the most useful for in higher dimensional change detection problems as well.

Thirdly, we measure sensitivity by the normal Hellinger distance between the marginal distributions of projections before and after a change,
while \citet{kuncheva_pca_2014} use the normal Bhattacharyya distance.
We refer to Section \ref{sec:problem} for an explanation of this choice.

In short, for bivariate data, we show that if only one mean changes in any direction, one variance increases, or the correlation changes, the minor projection is the most sensitive.
The principal projection is the most sensitive if one variance decreases and the correlation is not too close to $1$.
Lastly, if both means change, which projection is the most sensitive depends on the relative directions and sizes of change, 
and when both variances change by an equal amount, both projections are equally sensitive.
Thus, on average (with all change scenarios equally likely), the minor projection is the most sensitive, mainly due to the sparse changes.
Our simulations show that the trend of the minor projections being more sensitive on average also holds for higher dimensions.
This knowledge is important to create more efficient change or anomaly detection methods.

The rest of the article is organized as follows: Section \ref{sec:problem} formulates the problem precisely, 
Section \ref{sec:results} contains the theoretical results about sensitivity to changes in two dimensions, 
while in Section \ref{sec:explore_higher}, we explore sensitivity in higher dimensions by simulations.
The proofs for the results in Section \ref{sec:results} are given in Appendix \ref{sec:proofs}.

\section{Problem formulation} \label{sec:problem}
Consider $n$ independent observations $\mathbf{x}_t \in \mathbb{R}^D$.
For $t = 1, \ldots, \kappa$, where $1 < \kappa < n$, the data has mean $\bmu_0$ and
covariance matrix $\bSigma_0$, while for $t = \kappa + 1, \ldots, n$, the data has mean $\bmu_1$ and covariance matrix $\bSigma_1$.
Assume without loss of generality that the data is standardized with respect to the pre-change parameters, so that $\bmu_0 = \mathbf{0}$
and $\bSigma_0$ is a correlation matrix.
For $D = 2$, the changed mean is given by $\bmu_1 = (\mu_1, \mu_2)^\tp$ and the covariance matrices by
\[
  \bSigma_0 = 
  \begin{pmatrix}
    1 & \rho \\
    \rho & 1
  \end{pmatrix}
  \quad \text{and} \quad
  \bSigma_1 = 
  \begin{pmatrix}
    a_{11}^2 & a_{11}a_{22}a_{12}\rho \\
    a_{11}a_{22}a_{12}\rho & a_{22}^2
  \end{pmatrix}.
\]
where
\begin{equation}
  -1 < \rho,\; a_{12} \rho < 1 \quad \text{and}\quad \rho \not= 0.
  \label{eq:corBound}
\end{equation}
Here, $a_{11}$ and $a_{22}$ denote multiplicative change factors for the standard deviation. 
For example, if $a_{11} = 0.5$ the standard deviation is half of what it was originally.
Similarly, $a_{12}$ is the change factor for the correlation.
Note that we exclude the degenerate cases of correlations $-1$ and $1$.

Next, let $\{ \lambda_j, \bv_j \}_{j = 1}^D$ be the normalized eigensystem of $\bSigma_0$, ordered by $\lambda_1 \geq \ldots \geq \lambda_D$.
The orthogonal projections $y_{j, t} = \bv_j^\tp \bx_t$ are our main objects of interest.
We refer to $y_{j, t}$ as \textit{principal projections} when $j$ is close to $1$ and \textit{minor projections} if $j$ is close to $D$.
As this is meant as approximate terminology, we do not specify further what "close" means.
The general problem is to find out which of the projections that are the most sensitive to different distributional changes defined by $(\bmu_1, \bSigma_1)$, 
for each pre-change correlation matrix $\bSigma_0$.
In the bivariate case, $(\bSigma_0, \bmu_1, \bSigma_1)$ is fully specified by $(\rho, \mu_1, \mu_2, a_{11}, a_{12}, a_{22})$.

We define sensitivity to changes as the normal Hellinger distance between the marginal distribution of a projection before and after a change.
The squared Hellinger distance between two normal distributions $p(x) = N(x | \xi_1, \sigma_1)$ and $q(x) = N(x | \xi_2, \sigma_2)$ is given by
\[
  H^2(p, q) = 1 - \sqrt{\frac{2\sigma_1\sigma_2}{\sigma_1^2 + \sigma_2^2}} \exp \left\{ -\frac{1}{4} \frac{(\xi_1 - \xi_2)^2}{\sigma_1^2 + \sigma_2^2} \right\}.
\]
The formal definition of sensitivity to changes is contained in Definition \ref{def:sensitivity}.
\begin{mydef} \label{def:sensitivity}
For $j = 1, \ldots, D$, let $p_j$ and $q_j$ denote the marginal pre- and post-change density functions of $y_{j, t}$, respectively, given by
\begin{align*}
  p_j(y) &= N(y |\; \bv_j^\tp \bmu_0, \bv_j^\tp \bSigma_0 \bv_j) = N(y |\; 0, \lambda_j) \\
  q_j(y) &= N(y |\; \bv_j^\tp \bmu_1, \bv_j^\tp \bSigma_1 \bv_j).
\label{eq:projPdf}
\end{align*}
The \textit{sensitivity of the $j$'th projection based on $\bSigma_0$ to the change specified by $(\bmu_1, \bSigma_1)$} is defined as $H(p_j, q_j)$, abbreviated by $H_j$ or $H_j(\bSigma_0, \bmu_1, \bSigma_1)$.
\end{mydef}

In light of Definition \ref{def:sensitivity}, our aim in the next section is to study for which pre-change parameters and changes the inequality $H_2 > H_1$ holds when $D = 2$.

\begin{myremarks}
  (i) \citet{kuncheva_pca_2014} also define sensitivity as a divergence between distributions before and after a change, but use the Bhattacharyya distance.
  The closely related Hellinger distance was chosen here because it turned out to be simpler to prove the sensitivity propositions because of Lemma \ref{lem:HellingerEquiv}, found in the Appendix.
  It is also an advantageous feature of the Hellinger distance that it is a true metric and takes values in $[0, 1]$.
  That it is a true metric implies for instance that a change in variance from $1$ to $a > 1$ is an equally large change
  as from $1$ to $1/a$.
  We find this an appealing feature because it is also a property of the generalized likelihood ratio test for a change in
  the mean and/or variance of normal data (see \citet{hawkins_statistical_2005} for the corresponding test statistic).

  (ii) The difference between our approach and the work of \citet{kuncheva_pca_2014} can now be stated more clearly.
  Our aim is to study the sensitivity of the $y_{j, t}$'s as functions of parameters of the original data $\bx_t$.
  \citet{kuncheva_pca_2014}, on the other hand, study (additive) changes in the parameters of $\by_t$ directly;
  for instance, $\lambda_j$ changing to $\lambda_j + a$ for all $j$, but without relating this $a$ back to which $\bSigma_1$'s this change corresponds to.
\end{myremarks}

\section{Bivariate results}  \label{sec:results}
This section contains all the bivariate results about sensitivity to changes. The detailed proofs are given in Appendix \ref{sec:proofs}.

For changes in the mean in two-dimensional data, the following proposition gives the condition for determining which projection that is the most sensitive.
\begin{myprop} \label{prop:mean}
  Let $a_{11} = a_{22} = a_{12} = 1$ and $\mu_1, \mu_2 \in \mathbb{R}$ while not both being $0$ simultaneously (only the mean changes).
  $H_2 > H_1$ if and only if $(\mu_1 - \mu_2)^2 > \mu_1\mu_2 (4/|\rho| - 2)$.
  In particular, for $|\rho| \in (0, 1)$,
  \begin{enumerate}[label=(\roman*)]
    \item $H_2 > H_1$ if one of $\mu_1$ and $\mu_2$ is $0$ while the other is not (one mean changes).
    \item $H_2 > H_1$ if $\mu_1 = -\mu_2 = \mu \not= 0$ (equal changes in opposite directions).
    \item $H_2 < H_1$ if $\mu_1 = \mu_2 = \mu \not= 0$ (equal changes in the same direction).
  \end{enumerate}
\end{myprop}

When both variances change by the same amount, we get that both projections are equally sensitive no matter what the pre-change correlation or size of the change is.
\begin{myprop} \label{prop:2var}
  Let $\mu_1 = \mu_2 = 0$, $a_{12} = 1$ and $a_{11} = a_{22} = a \not= 1$ (both variances change equally).
  For any $|\rho| \in (0, 1)$ and $a > 0$, $H_2 = H_1$.
\end{myprop}

The picture becomes more complicated when only one variance changes.
If the variance increases, the minor projection is always the most sensitive.
On the other hand, if the variance decreases, the principal projection is mostly the most sensitive,
but not always if the pre-change correlation is high ($\gtrsim 0.87$).
In total, this gives a slight edge to the minor projection.
\begin{myprop} \label{prop:1var}
  Let $\mu_1 = \mu_2 = 0$, $a_{12} = 1$, and either $a_{11} = 1$ and $a_{22} = a \not= 1$, or $a_{11} = a$ and $a_{22} = 1$, where $a > 0$ (one variance changes).
  \begin{enumerate}[label=(\roman*)]
    \item For any $|\rho| \in (0, 1)$ and $a > 1$ (variance increase), $H_2 > H_1$.
    \item When $|\rho| \in (0, 1)$ and $a \in (0, 1)$ (variance decrease), $H_2 < H_1$ in most cases. 
    The only exception is if $|\rho| \in (\sqrt{3}/2, 1)$ and $a \in (0, \sqrt{4\rho^2 - 3})$, where $H_2 > H_1$.
  \end{enumerate}
\end{myprop}

For a change in correlation, the minor projection is the most sensitive in most cases.
Only if the correlation changes direction and becomes stronger is the principal projection more sensitive.
\begin{myprop} \label{prop:cor}
  Let $\mu_1 = \mu_2 = 0$, $a_{11} = a_{22} = 1$ and $a_{12} = a \not= 1$ such that \eqref{eq:corBound} holds (the correlation changes).
  Then $H_2 > H_1$ for any $|\rho| \in (0, 1)$ and $a > -1$.
\end{myprop}

\section{Exploring higher dimensions} \label{sec:explore_higher}
In the two-dimensional case, we saw that which projection is the most sensitive depends both on the change $(\bmu_1, \bSigma_1)$ as well as the pre-change correlation matrix $\bSigma_0$.
For higher dimension $D$, solving inequalities like above for all the parameters in $(\bSigma_0, \bmu_1, \bSigma_1)$ quickly becomes extremely tedious.
Therefore, we use simulation to obtain Monte Carlo estimates $E[H_j(\bSigma_0, \bmu_1, \bSigma_1)]$ in stead,
where we vary which parameters that change, the size of the changes and the sparsity of the change (the number of dimensions that change).
Let $\rho_{i, d}$ for $i \not= d$ denote the off-diagonal elements of $\bSigma_0$, $\mu_{d}$ be the $d$'th element of $\bmu_1$, and $\sigma_{d}$ be the $d$'th diagonal element of $\bSigma_1$.
Then our simulation protocol to get such estimates is as follows:
\begin{enumerate}
  \item Draw a correlation matrix $\bSigma_0$ uniformly from the space of correlation matrices by the method of \citet{joe_generating_2006} (clusterGeneration::rcorrmatrix in R).
  \item Draw a change sparsity $K \sim \Unif\{ 2, \ldots, D  \}$.
  \item Draw a random subset $\mathcal{D} \subseteq \{ 1, \ldots, D \}$ of size $K$.
  \item Draw an additive change in mean $\mu \sim \Unif(-3, 3)$, and set $\mu_{d} = \mu$ for $d \in \mathcal{D}$, while $\bSigma_1 = \bSigma_0$.
  \item Draw a multiplicative change in standard deviation $\sigma \sim \frac{1}{2}\Unif(1/3, 1) + \frac{1}{2}\Unif(1, 3)$ (equal probability of decrease and increase in standard deviation)
  and set $\sigma_{d} = \sigma$ for $d \in \mathcal{D}$, keeping the remaining parameters constant.
  \item Draw a multiplicative change in correlation $a \sim \Unif(0, 1)$ and change $\rho_{i, d}$ to $a\rho_{i, d}$ for all $i \not= d \in \mathcal{D}$. The other parameters are kept constant.
  \item For each of the three change scenarios 4-6, calculate $H_j(\bSigma_0, \bmu_1, \bSigma_1)$, $j = 1, \ldots, D$.
  \item Repeat 2-7 $10^3$ times.
  \item Repeat 1-8 $10^3$ times.
\end{enumerate}
Averaging the simulated $H_j$'s yields estimates of $E[H_j]$, and we can condition on the type of parameter that change and the change sparsity to see what the sensitivity is expected to be
for different classes of changes.
(Note that we only consider decreases in correlation. This is to avoid getting too many indefinite $\bSigma_1$'s.
If indefinite $\bSigma_1$'s still occur, we find the closest positive-definite one by the Matrix::nearPD-function in R and set the diagonal to 1.)

Figure \ref{fig:hellinger} shows that the trend of the minor components being the most sensitive continues for $D = 20$ and $D = 100$.
This holds for changes in the mean, variance and correlation as well as all the different change sparsities.
From the quantile plots, however, observe that a lot of variation is hidden in these averages, 
meaning that which projection is the most sensitive will depend on the specific $\bSigma_0$ and change $(\bmu_1, \bSigma_1)$, as in the bivariate case.

\begin{figure}[ht]
  \centering
  \includegraphics[scale = 0.7]{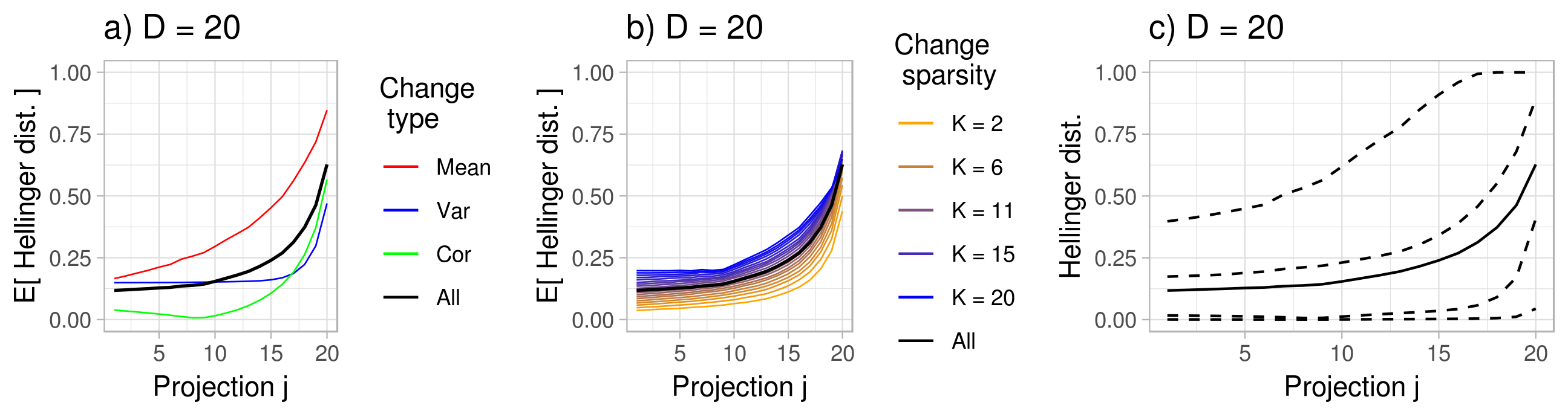}
  \includegraphics[scale = 0.7]{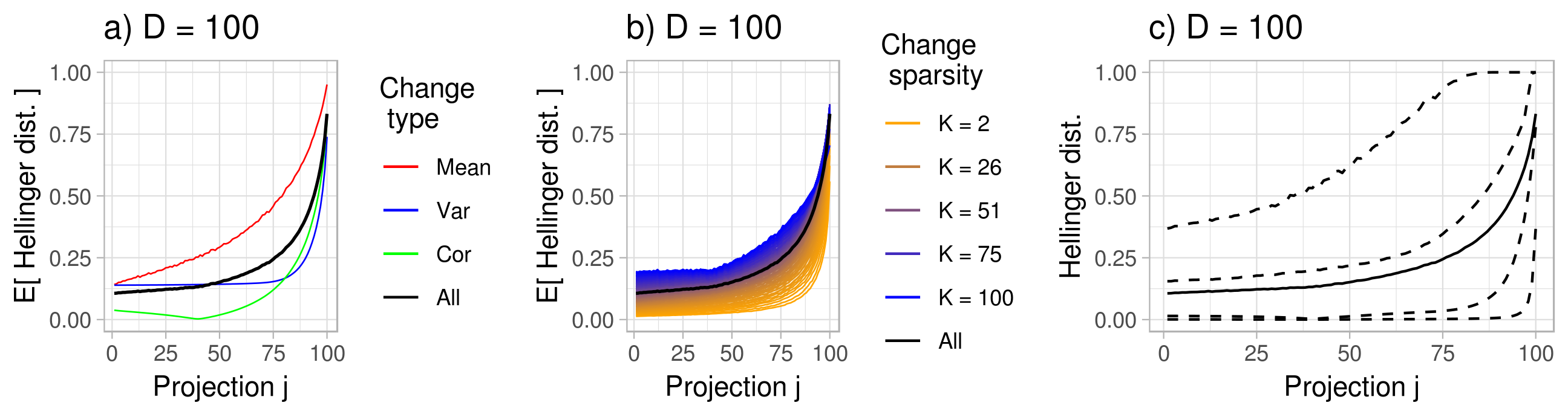}
  \caption{A summary of the sensitivity results obtained by the simulation protocol for $D = 20$ for $D = 100$.
  a) Monte Carlo estimates of $E[H_j]$ for uniformly drawn changes in the mean, variance and (decreases in) correlation, as well as uniformly drawn pre-change correlation matrices $\bSigma_0$.
  b) Same as a), but now the average sensitivity is conditional on the sparsity of the change, rather than the type of parameter.
  c) 0.05, 0.25, 0.75 and 0.95 percentiles (the dashed lines, from bottom to top) of the distribution of $H_j$, together with $E[H_j]$ (solid line).
  Note that the percentiles are over $\bSigma_0, \bmu_1$ and $\bSigma_1$ simultaneously.}
  \label{fig:hellinger}
\end{figure}

\section{Concluding remarks}
We have presented bivariate theory showing that the minor projection of PCA-rotated data is usually the most sensitive to changes, especially if the change is sparse.
Simulations also confirm this to be the case on average for higher dimensions, but also that the sensitivity strongly varies with the pre-change correlation matrix and the specific change.
In future work it will be interesting to exploit this knowledge about sensitivity for specific change detection tasks.

\section*{Acknowledgements}
This work is partially funded by the Norwegian Research Council centre Big Insight, project 237718.

\printbibliography

\setcounter{section}{0}
\renewcommand{\thesection}{\Alph{section}}
\section{Appendix: Proofs} \label{sec:proofs}
Before turning to the the proofs of the propositions in section \ref{sec:results},
the expressions for the pre- and post-change means and variances of each projection is needed.
The normalized eigenvectors (principal axes) and corresponding eigenvalues (variance in the data along a given principal axis) of 
$\bSigma_0$ are quickly verified to be
\begin{equation}
\begin{split}
&\lambda_1 = 1 + \rho, \quad
\mathbf{v}_1 = \frac{1}{\sqrt{2}}
\begin{pmatrix}
1 \\
1
\end{pmatrix} \\
&\lambda_2 = 1 - \rho, \quad
\mathbf{v}_2 = \frac{1}{\sqrt{2}}
\begin{pmatrix}
-1 \\
1
\end{pmatrix}
\end{split}
\label{eq:eigenSpace}
\end{equation}
Note that which principal axis is the dominant one depends on the sign of $\rho$.
If $\rho$ is positive, $\mathbf{v}_1$ is the dominant one, but $\mathbf{v}_2$ is dominant if $\rho$ is negative.

From the projections in \eqref{eq:eigenSpace}, the parameters of the projections before and after a change can be expressed as functions of the original correlation matrix and multiplicative change factors.
For the principal component, the original and changed variances become as follows, respectively.
\begin{equation}
\begin{split}
o_1^2 &= 1 + \rho \\
c_1^2 &= \frac{1}{2}a_{11}^2 + \frac{1}{2}a_{22}^2 + a_{11}a_{22}a_{12}\rho.
\end{split}
\label{eq:topPCvariances}
\end{equation}
The expressions for the variances of the minor component are identical up to one switched sign;
\begin{equation}
\begin{split}
o_2^2 &= 1 - \rho \\
c_2^2 &= \frac{1}{2}a_{11}^2 + \frac{1}{2}a_{22}^2 - a_{11}a_{22}a_{12}\rho.
\end{split}
\label{eq:botPCvariances}
\end{equation}
Observe that if $\rho < 0$, then $o_2$ and $c_2$ would be equal to $o_1$ and $c_1$ with positive $\rho$, and vice versa.
Thus, for $\rho \in (-1, 1)$, the general expressions are obtained by replacing $\rho$ with $|\rho|$.
Lastly, the changed mean components are given by
\begin{equation}
\begin{split}
m_1 &= \frac{1}{\sqrt{2}}(\mu_{1} + \mu_{2}) \\
m_2 &= \frac{1}{\sqrt{2}}(\mu_{1} - \mu_{2}).
\end{split}
\label{eq:PCmeans}
\end{equation}

We first prove Proposition \ref{prop:mean} for changes in the mean. 
\begin{proof}[Proof of Proposition \ref{prop:mean}]
Let $p_1(x) = N(x | 0, o_1^2)$, $q_1(x) = N(x | m_1, o_1^2)$,
$p_2(x) = N(x | 0, o_2^2)$ and $q_2(x) = N(x | m_2, o_2^2)$, where $m_i, o_i$ are as in \eqref{eq:topPCvariances}, \eqref{eq:botPCvariances} and \eqref{eq:PCmeans}, with $\rho$ replaced by $|\rho|$ as noted above.
The Hellinger distances between the distributions before and after a change along each principal axis is given by
\[
  H_j^2 = H^2(p_j, q_j) = 1 - \exp\left\{ - \frac{1}{8o_j^2}m_j^2 \right\}.
\]
Then some algebra results in the inequality we needed to prove;
\begin{align*}
  H_2 &> H_1 \\
  (1 + |\rho|)(\mu_1 - \mu_2)^2 &> (1 - |\rho|)(\mu_1 + \mu_2)^2 \\
  |\rho| (\mu_1 - \mu_2)^2 + \mu_1\mu_2(2|\rho| - 4) &> 0 \\
  (\mu_1 - \mu_2)^2 &> \mu_1\mu_2 (4/|\rho| - 2).
\end{align*}
From this inequality, the three special cases (i), (ii) and (iii) are immediately given.
\end{proof}

In the proofs concerning changes in the covariance matrix, we will make use of the following lemma.
It reduces the inequality of Hellinger distances to a simpler inequality of ratios of variances.
\begin{mylemma} \label{lem:HellingerEquiv}
  Let $p_1, q_1, p_2, q_2$ be 0-mean normal distribution functions with variances $\sigma^2_{p_1}, \sigma^2_{q_1}, \sigma^2_{p_2}$ and $\sigma^2_{q_2}$, respectively.
  Furthermore, let
  \[
    \log r_j = \left|\log \frac{\sigma^2_{q_j}}{\sigma^2_{p_j}}\right|, \quad j = 1, 2.
  \]
  Then $H(p_2, q_2) > H(p_1, q_1)$ if and only if $\log r_2 > \log r_1$.
\end{mylemma}
\begin{proof}
  First observe that when the means are $0$ we can write the Hellinger distance between two normal distributions as the following.
  \begin{align*}
    H^2 (p, q) &= 1 - \left( \frac{2\sigma_p\sigma_q}{\sigma_p^2 + \sigma_q^2} \right)^{1/2} \\
    &= 1 - \sqrt{2}\left( \frac{\sigma_p}{\sigma_q} + \frac{\sigma_q}{\sigma_p} \right)^{-1/2} \\
    &= 1 - \sqrt{2}\left( \frac{\sigma_p^2}{\sigma_q^2} + \frac{\sigma_q^2}{\sigma_p^2} + 2 \right)^{-1/4}.
  \end{align*}
  This gives us the inequality
  \begin{align*}
    H(p_2, q_2) &> H(p_1, q_1) \\
    \frac{\sigma_{p_2}^2}{\sigma_{q_2}^2} + \frac{\sigma_{q_2}^2}{\sigma_{p_2}^2} &> \frac{\sigma_{p_1}^2}{\sigma_{q_1}^2} + \frac{\sigma_{q_1}^2}{\sigma_{p_1}^2}.
  \end{align*}
  By setting $r_2 = \sigma_{p_2}^2 / \sigma_{q_2}^2$ and $r_1 = \sigma_{p_1}^2 / \sigma_{q_1}^2$, the inequality can be written as
  \[
    r_2 + r_2^{-1} > r_1 + r_1^{-1}.
  \]
  Now assume first that $r_1, r_2 > 1$, i.e., $\sigma_{p_j}^2 > \sigma_{q_j}^2$.
  Then we see that
  \begin{align*}
    r_2 + r_2^{-1} &> r_1 + r_1^{-1} \\
    r_2 - r_1 + \frac{r_1 - r_2}{r_1r_2} &> 0 \\
    (r_2 - r_1)\Big(1 - \frac{1}{r_1r_2}\Big) &> 0.
  \end{align*}
  By the assumption that $r_1, r_2 > 1$, this inequality holds if and only if $r_2 > r_1$.
  
  Finally, note that by interchanging $\sigma_{p_j}^2$ and $\sigma_{q_j}^2$, the same result is obtained when $\sigma_{q_j}^2 \geq \sigma_{p_j}^2$.
  Thus, to make the result hold in general, we can set
  \[
    r_j = \exp\left\{\Big|\log \frac{\sigma^2_{q_j}}{\sigma^2_{p_j}}\Big|\right\}, \quad j = 1, 2,
  \]
  which is an expression for the ratio between variances where the largest of the variances is always in the numerator.
  Therefore we get that $\log r_2 > \log r_1$ is equivalent to $H_2 > H_1$.
\end{proof}

The rest of this article contain the individual proofs of the remaining propositions in the main body of the text.
%
\begin{proof}[Proof of Proposition \ref{prop:2var}]
When assuming that $a_{12} = 1$ and $a_{11} = a_{22} = a \not= 1$, we get that
\[
  \log r_2 = \left| \log \frac{a^2/2 + a^2/2 - |\rho| a^2}{1 - |\rho|} \right| = |\log a^2|,
\]
and
\[
  \log r_1 = \left| \log \frac{a^2/2 + a^2/2 + |\rho| a^2}{1 + |\rho|} \right| = |\log a^2|.
\]
Hence, by arguments similar to the proof of Lemma \ref{lem:HellingerEquiv}, we see that $H_2 = H_1$ no matter what $|\rho|$ or $a$ is.
\end{proof}

%
\begin{proof}[Proof of Proposition \ref{prop:1var}]
  Using the formulas for the variances of the projections \eqref{eq:topPCvariances} and \eqref{eq:botPCvariances},
  the inequality we have to study according to Lemma \ref{lem:HellingerEquiv} becomes the following,
  \begin{align}
    \left| \log \frac{a^2 - 2a|\rho| + 1}{2(1 - |\rho|)} \right| &> \left| \log \frac{a^2 + 2a|\rho| + 1}{2(1 + |\rho|)} \right| \notag \\
    \left| \log \left[ \frac{(1 - a)^2}{2(1 - |\rho|)} + a \right] \right| &> \left| \log \left[ \frac{(1 - a)^2}{2(1 + |\rho|)} + a \right] \right|. \label{eq:1var_ineq}
  \end{align}
  First, we have to find out the sign of the expressions inside the absolute values for each $a$ and $|\rho|$.
  For the left-hand side, we get
  \begin{align*}
    &\frac{(1 - a)^2}{2(1 - |\rho|)} + a = 1 \\
    &a = 1 \text{ and } a = 2|\rho| - 1.
  \end{align*}
  Thus, for $a > 1$ and $a < 2|\rho| - 1$, the left-hand side is positive, while negative in between.
  For the right-hand side, the expression inside the absolute value signs are positive for $a > 1$ and $a < - (1 + 2|\rho|)$. 
  Since $a > 0$, however, the relevant root for the right-hand side is only $a = 1$.
  In total, this gives us three regions of $(a, |\rho|)$-values to check inequality \eqref{eq:1var_ineq}: $a > 1$ and $|\rho| \in (0, 1)$, $a \in (2|\rho| - 1, 1)$ and $|\rho| \in (0, 1)$, and $a \in (0, 2|\rho| - 1)$ and $|\rho| \in (1/2, 1)$.
  \paragraph{$a > 1$ and $|\rho| \in (0, 1)$:}
  The absolute value signs can now be dissolved, so that inequality \eqref{eq:1var_ineq} becomes
  \[
    \frac{(1 - a)^2}{(1 - |\rho|)} > \frac{(1 - a)^2}{(1 + |\rho|)}.
  \]
  Since $|\rho| \in (0, 1)$, we see that the inequality holds for any $a > 1$.
  Hence, $H_2 > H_1$ in this scenario; when the variance increases.
  \paragraph{$a \in (2|\rho| - 1, 1)$ and $|\rho| \in (0, 1)$:}
  In this case, inequality \eqref{eq:1var_ineq} becomes
  \[
    \frac{(1 - a)^2}{(1 - |\rho|)} < \frac{(1 - a)^2}{(1 + |\rho|)}.
  \]
  I.e., it does not hold for any of the $a$'s or $|\rho|$'s within the relevant region. Note that when $|\rho| < 1/2$, $a$ is kept between $(0, 1)$.
  \paragraph{$a \in (0, 2|\rho| - 1)$ and $|\rho| \in (1/2, 1)$:}
  Now we get the inequality
  \[
    \frac{(1 - a)^2}{2(1 - |\rho|)} + a > \left(\frac{(1 - a)^2}{2(1 + |\rho|)} + a\right)^{-1},
  \]
  which is equivalent to
  \begin{equation}
    a^4 - a^2(4\rho^2 - 2) + 4\rho^2 - 3 > 0.
    \label{eq:1var_tedious}
  \end{equation}
  The roots of the function on the left-hand side are $a = \pm 1$ and $a = \pm \sqrt{4\rho^2 - 3}$, but the only relevant root for $a \in (0, 2|\rho| - 1)$ and $|\rho| \in (1/2, 1)$ is $a_0 := \sqrt{4\rho^2 - 3}$.
  
  Next, for $|\rho| < \sqrt{3}/2$ the root $a_0$ moves into the complex plane, and the function on the left-hand side of \eqref{eq:1var_tedious} is always less than $0$ for the relevant $a$'s. I.e., $H_2 < H_1$ in this case.
  If $|\rho| > \sqrt{3}/2$, on the other hand, then \eqref{eq:1var_tedious} holds for $a \in (0, a_0)$, but not for $a \in (a_0, 2|\rho| - 1)$.
\end{proof}
%

\begin{proof}[Proof of Proposition \ref{prop:cor}]
  In this scenario, the inequality to check due to Lemma \ref{lem:HellingerEquiv} and expressions \eqref{eq:topPCvariances} and \eqref{eq:botPCvariances} is
  \begin{equation}
    \left| \log \frac{1 - a|\rho|}{1 - |\rho|} \right| > \left| \log \frac{1 + a|\rho|}{1 + |\rho|} \right|.
    \label{eq:cor_ineq}
  \end{equation}
  To dissolve the absolute value signs we first have to see for which values of $a$ and $|\rho|$ the expressions inside are positive or negative. It is easily verified that the expression inside the left-hand side absolute value is positive for $a < 1$, while the right-hand side is positive if $a > 1$, both being negative otherwise.
  
  First assume that $a < 1$. Then inequality \eqref{eq:cor_ineq} becomes 
  \begin{align*}
    \frac{1 - a|\rho|}{1 - |\rho|} &> \frac{1 + |\rho|}{1 + a|\rho|} \\
    1 - (a\rho)^2 &> 1 - \rho^2 \\
    a^2 &< 1
  \end{align*}
  Hence, $a \in (-1, 1)$ yields $H_2 > H_1$.
  On the other hand, if $a > 1$, we obtain
  \begin{align*}
    \frac{1 - |\rho|}{1 - a|\rho|} &> \frac{1 + a|\rho|}{1 + |\rho|} \\
    a^2 &> 1,
  \end{align*}
  which is always true.
  Thus, in total, $H_2 < H_1$ only if $a < -1$.
\end{proof}

\end{document}